\documentclass{mathincs}
\usepackage{amsrefs}
\usepackage{epsf,latexsym,graphicx}
\usepackage[matrix,arrow,tips,curve]{xy}
\usepackage{color} 
\usepackage[usenames,dvipsnames,svgnames,table]{xcolor}
\usepackage{tikz}
\usetikzlibrary{arrows,matrix,positioning}
\usepackage{etex}
\usepackage{xy}
\usepackage{amssymb,amsthm}
\usepackage{enumerate}
\usepackage{verbatim}
\usepackage{epsfig}
\usepackage{pst-grad}
\usepackage{pst-plot}
\usepackage{cases}
\usepackage{hyperref}
 \newtheorem{thm}{Theorem}[section]
 \newtheorem{cor}[thm]{Corollary}
 \newtheorem{lem}[thm]{Lemma}
 \newtheorem{prop}[thm]{Proposition}
 \theoremstyle{definition}
 
 \theoremstyle{remark}
 \newtheorem{rem}[thm]{Remark}
 \newtheorem{ex}{Example}
 \numberwithin{equation}{section}
\newcommand{\R}{{\ensuremath{\mathbb{R}}}}
\newcommand{\pr}{{\ensuremath{\mathbb{P}}}}
\newcommand{\calr}{\mathcal{R}}
\newcommand{\J}{{\ensuremath{\mathcal{J}}}}
\renewcommand{\Im}{\mathrm{Im\,}}
\newcommand\ST{\rule[-1em]{0pt}{1.49em}}
\newcommand\TT{\rule[-0.5em]{0pt}{1.0em}}

\begin{document}
%
%
%
\title[Bounds on the dimension of trivariate spline spaces]
 {Bounds on the dimension of trivariate spline spaces: \\
 A homological approach}
\author{Bernard Mourrain}

\address{%
Inria Sophia Antipolis M\'editerran\'ee\\
2004 route des Lucioles, B.P.~93\\
06902 Sophia Antipolis, France}

\email{Bernard.Mourrain@inria.fr}

\author{Nelly Villamizar}
\address{Johann Radon Institute for Computational and Applied Mathematics (RICAM)\br
Austrian Academy of Sciences\br
Altenberger Stra{\ss}e 69\br
4040 Linz, Austria }
\email{nelly.villamizar@ricam.oeaw.ac.at}


\date{October 15, 2013}

\begin{abstract} 
We consider the vector space of globally differentiable piecewise polynomial functions defined on a three-dimensional polyhedral domain partitioned into tetrahedra. 
We prove new lower and upper bounds on the dimension of this space by applying homological techniques. We give an insight into different ways of approaching this problem by exploring its connections with the Hilbert series of ideals generated by powers of linear forms, fat points, the so-called Fr\"oberg--Iarrobino conjecture, and the weak Lefschetz property.
\end{abstract}

\maketitle

\section{Introduction}
A \textit{spline} is a function which is conformed by pieces of polynomials defined on a rectilinear partition of a domain in the $d$-dimensional real space, and joined together to ensure some degree of global smoothness.
For a tetrahedral partition $\Delta$ of a domain embedded in $\R^3$, we denote by $C_k^r(\Delta)$ the space of {splines} or piecewise polynomial functions  of degree less than or equal to $k$ defined on $\Delta$, with global order of smoothness $r$ ($\geq 0$). This set is a vector space over $\R$. We usually refer to it as the space of $C^r$ trivariate splines of degree $k$ on $\Delta$.

Trivariate spline spaces are important tools in approximation theory and numerical analysis; they have been used, for instance, to solve boundary value problems by the finite-element method (FEM) \cites{zlamal,Zen} (see also \cite{LaiSchu} and the references therein). 
More recently, they have became highly effective tools in Isogeometric Analysis \cite{hughes}, 
which is a recently developed computational approach combining the exact topology description in finite element analysis, with the accurate shape representation in Computer Aided Design (CAD). 
In these areas of application, finding the dimension for $C_k^r(\Delta)$ on general tetrahedral partitions is a major open problem. It has been studied using Bernstein--B\'ezier methods in articles by Alfeld, Schumaker, et al. \cites{ASW,AfSS,boundsTriva}. The results in these papers do not take into account the geometry of the faces surrounding the interior edges or interior vertices.  A variant of that approach  by Lau  \cite{lau}, gives a lower bound for simply connected tetrahedral partitions; the formula, although it contains a term which takes into account the geometry of faces surrounding interior edges, lacks of the term involving the number of interior vertices. This  frequently makes the lower bound much smaller than the one presented in \cite{ASW}. In general, an arbitrary small change of the location of the vertices can change the dimension of $C_k^r(\Delta)$. Thus, finding an accurate way of including conditions in the formulas regarding the geometry of the configuration is one of the main open problems in trivariate spline theory.

The use of homological algebra in solving the dimension problem on spline spaces dates back to 1988, \cite{bil}. In this article, Billera considers triangulated $d$-dimensional regions in $\R^d$; with his approach he gave the first proof of the generic dimension of the space of $C^1$ bivariate splines. In \cite{family}, Schenck and Stillman introduced a different chain complex which, in the bivariate settings, leads 
to prove a new formula for the upper bound on the dimension of the space, generalizing the bounds already known \cite{LaiSchu}, and also yielding a simple proof of the dimension formula for the space of $C^r_{k}$ splines with degree $k\geq 4r+1$ (\cite{bn1}).  
Our aim in this paper is to apply the homological approach to approximate the dimension of trivariate spline spaces.

The formulas we present, apply to any tetrahedral partition $\Delta$, any degree $k$, and any order of global smoothness $r$. They include terms that depend on the number of different planes surrounding the edges and vertices in the interior of $\Delta$.  The main contribution of this paper are the new formulas for lower and upper bounds on the dimension of the space of $C^r$ trivariate splines. These bounds represent an improvement with respect to previous results in the literature. Moreover, the construction 
and results 
we present throughout the paper give an insight into ways of improving the bounds and finding the exact dimension under certain conditions. 

The paper is organized as follows. In Section \ref{construction}, we
recall the chain complex of modules  proposed by Schenck \cite{Sck}, in the case of a three-dimensional simplicial complex.
The Euler characteristic equation for such chain complex yields a formula for $\dim C_k^r(\Delta)$ in terms of the dimension of the modules and homology modules of the complex. In order to get an explicit expression for the dimension, in terms of known information on the subdivision, there are two important aspects to consider.
Firstly, we need to analyze ideals generated by powers of linear forms in three variables, and no general resolution is known  for this kind of ideals as it was in the case of ideals in two variables \cite{family}. Secondly, it is also necessary to determine kernels of maps to compute the dimension of the two homology modules appearing in the formula. These two problems are considered as follows. 
In  Section \ref{dimJ0}, by using the Fr\"oberg sequence associated to
an ideal generated by homogeneous polynomials of prescribed degrees,
we obtain a formula that approximates the dimension of ideals
generated by powers of linear forms. We discuss the cases where this
formula gives the exact dimension, and some other formulas that can be
used to get better bounds when the number of linear forms are less
than nine. This discussion covers the relationship between splines and fat points \cite{iarrobino}, and connections of this theory with the Weak Lefschetz Property, Hilbert series of ideals of powers of generic linear forms, and Fr\"oberg's conjecture and its most recent versions. In Section \ref{bounds1}, we 
rewrite the  dimension formula obtained in Section \ref{construction} and prove an explicit upper bound on $\dim C_k^r(\Delta)$ for any $\Delta$. Additionally, we prove an upper bound that can be applied in the free case (when the first two homology modules vanish). Similarly, a lower bound is proved in Section \ref{bounds2}. Finally, we conclude with some examples in Section \ref{examples}.

The approaches we make to the problem differ from the ones used before, see \cite{LaiSchu} and the references therein, hence enriching the tools to confront the problem, and thus to develop the theory of trivariate splines. 

\section{Construction of the chain complex}\label{construction}

The notations and definitions we present in this section appear in \cite{Local} for a finite $d$-dimensional simplicial complex, here we restrict them to the trivariate case. 

For a 3-dimensional simplicial complex  $\Delta$, supported on $|\Delta|\subset\R^3$, such that $|\Delta|$ is homotopy equivalent to a  3-dimensional ball, let $\Delta^0$ and $\Delta_{i}^0$ ($i=0,1,2,3$) be the set of interior faces, and  $i$-dimensional interior faces of $\Delta$  whose support is not contained in the boundary of  $|\Delta|$. Denote by $f^0$ and $f_i^0$ ($i=0,1,2,3$)  the cardinality of these sets, respectively. Let $\partial\Delta$ be the boundary complex consisting of all $2$-faces lying on just one $3$-dimensional face as well as all subsets of them.

We will study the dimension of the vector space $C_k^r(\Delta)$ by studying its ``homogenization" defined on a similar partition but in one dimension higher. Let us denote by $\hat\Delta$ the simplex obtained by embeding $\Delta$ in the hyperplane $\{w=1\}\subseteq\R^4$ forming a cone  over $\Delta$ with vertex at the origin. Denote by $C^r(\hat{\Delta})_k$ the set of $C^r$ splines on $\hat{\Delta}$ of degree exactly $k$. Then $C^r(\hat{\Delta}):=\oplus_{\geq 0}C^r(\hat{\Delta})_k$ is a  graded $\R$-algebra. Furthermore, the elements of $C^r(\hat\Delta)_k$ are the homogenization of the elements of $C_k^r(\Delta)$, so there is a vector space isomorphism between $C_k^r(\Delta)$ and $C^r(\hat\Delta)_k$, and  
in particular
\begin{equation}\label{dimensionequality}
\dim C_k^r(\Delta)=\dim C^r(\hat{\Delta})_k.
\end{equation}
Thus, to study the dimension of the space $C_k^r(\Delta)$ it suffices to study the Hilbert series of the module $C^r(\hat\Delta)$, which is a finitely generated graded module.  

Let us define $R := \R[x, y, z, w]$. For every $2$-dimensional face $\sigma\in\Delta_2^0$, let $\ell_\sigma$ denote the homogeneous linear form vanishing on $\hat\sigma$ (this is just the homogenization of the linear equation vanishing on $\sigma$),
and define the chain complex $\J$ of ideals of $R$  as follows:
\begin{align*}
\J(\iota)&=\langle 0\rangle&&\text{for each} \;\iota\in\Delta_3^0\\
\J(\sigma)&=\langle \ell_{\sigma}^{r+1}\rangle &&\text{for each} \;\sigma\in\Delta_2^0\\
\J(\tau)&=\langle \ell_{\sigma}^{r+1}\rangle_{\sigma\ni\tau} &&\text{for each} \;\tau\in\Delta_1^0, \, \sigma\in\Delta_2^0\\
\J(\gamma)&=\langle \ell_{\sigma}^{r+1}\rangle_{\sigma\ni\gamma}&&\text{for each} \;\gamma\in\Delta_0^0, \,\sigma\in\Delta_{2}^0.
\end{align*}
Let $\mathcal{R}$ be the constant complex on $\Delta^0$,  defined by
$\mathcal{R}(\beta):=R$ \, for every $\beta\in\Delta^0$. Take
$\partial_i$ for $i=1,2,3$ to be the simplicial boundary maps from
$\calr_i\to\calr_{i-1}$ relative to $\partial\Delta$, where
$\calr_i:=\oplus_{\beta\in\Delta_i^0}\calr(\beta)$. Then
$H_i(\calr)=H_i(\Delta,\partial\Delta;R)$ is the relative simplicial
homology module, with coefficients in $R$. Let us consider the chain
complex $\calr/\J$ defined as the  quotient of $\calr$ by $\J$ ($\calr/\J(\beta)= \calr(\beta)/\J(\beta)$):
\vspace{0.1cm}
{
\[0\rightarrow\bigoplus_{\iota\in\Delta_3}\calr(\iota)\xrightarrow{\partial_3}\bigoplus_{\sigma\in\Delta_2^0}\mathcal{R}/\J(\sigma)\xrightarrow{\partial_2}\bigoplus_{{\tau\in\Delta}_1^{0}}\mathcal{R/J}(\tau)\xrightarrow{\partial_1}\bigoplus_{{\gamma\in\Delta}_0^{0}}\mathcal{R/J}(\gamma)\rightarrow 0\]}\vspace{0.1cm}where the maps $\partial_i$ are induced by the maps on $\calr_i$.
The top homology module  \[H_{3}(\mathcal{R/J}):=\ker(\partial_3)\]
is precisely the module $C^{r+1}(\hat\Delta)$ \cite{Sck}. Thus, by (\ref{dimensionequality}) we have that $\dim C_k^r(\Delta)=\dim H_3(\calr/\J)_k$, and the Euler characteristic equation applied to $\calr/\J$
\[\chi(H(\mathcal{R/J}))=\chi(\mathcal{R/J})\]
leads to the formula:
\begin{align}\label{firstdimformula}
\dim C^r_k(\Delta)=\sum_{i=0}^3 (-1)^i\hspace{-0.1cm}
\sum_{\beta\in\Delta_{3-i}^0}\dim \mathcal{R}/\mathcal{J}(\beta)_k + \dim H_{2}(\calr/\J)_k - \dim H_{1}(\calr/\J)_k.
\end{align}
The subindex $k$ means that we are considering the modules in degree exactly $k$. The goal is to determine the dimension of the modules in the previous formula as functions of known information about the subdivision $\Delta$.

Let us consider the short exact sequence of complexes
\[0\longrightarrow\mathcal{J}\longrightarrow\mathcal{R}\longrightarrow\mathcal{R/J}\longrightarrow 0.\]
It gives rise to the following long exact sequence of homology modules
\begin{align}\label{longsequence}
0\rightarrow &H_3(\mathcal{R})\rightarrow H_3(\mathcal{R}/\J)\rightarrow H_2 (\J)\rightarrow H_2(\mathcal{R})\rightarrow H_2(\mathcal{R}/\J)\rightarrow H_1(\J)\\
&\hspace{0.3cm}\rightarrow H_1(\mathcal{R})\rightarrow H_1 (\mathcal{R}/\J)\rightarrow H_0(\J)\rightarrow H_0(\mathcal{R})\rightarrow H_0 (\mathcal{R}/\J)\rightarrow 0\nonumber
\end{align}
Since by hypothesis $\Delta$ is supported on a ball, $H_i(\calr)$ is nonzero only for $i=3$ \cite{hatcher}*{Chapter 2}. Thus, by the long exact sequence (\ref{longsequence}), we have the following:
\begin{enumerate}[(i)]
\item \ST $H_{0}(\mathcal{R/J})=0$,
\item \ST $H_{1}(\mathcal{R/J})\cong H_{0}(\mathcal{J})$,\label{h1isoh0}
\item \ST $H_{2}(\mathcal{R/J})\cong H_{1}(\mathcal{J})$,
\item  $H_3(\calr)=R$, and hence $C^r(\Delta)\cong R\oplus H_2(\J)$. \label{iso4}
\end{enumerate}
Here the notation ``$A\cong B$" means that $A$ and $B$ are isomorphic as $R=\R[x,y,z,w]$ modules.
Let us notice that, in particular, the isomorphism in (\ref{iso4}) says that the study of the spline module reduces to the study of $H_2(\J)=\ker \partial_2$.

The complex of ideals $\J$, as defined above, is given by
\begin{equation}\label{Jcomplex}
0\longrightarrow\bigoplus_{\sigma\in\Delta_2^0}\J(\sigma)\xrightarrow{\partial_2}\bigoplus_{\tau\in\Delta_1^0}\J(\tau)\xrightarrow{\partial_1}\bigoplus_{\gamma\in\Delta_0^0}\J(\gamma)\longrightarrow 0
\end{equation}
where
$\partial_i$ are the restriction of the maps from the chain complex $\calr$.
Let us denote  $K_i:=\ker(\partial_i)$ and  $W_i:=\Im(\partial_{i+1})$ for $i=0,1$. Then, the homology modules $H_0(\J)$ and $H_1(\J)$ are by definition
\begin{align*}
H_0(\J)&:=\bigoplus_{\gamma\in\Delta_0^0}\J(\gamma)/W_0\\
H_1(\J)&:=K_1/W_1.
\end{align*}
By the short exact sequence
\[0\longrightarrow K_1\longrightarrow\bigoplus_{\tau\in\Delta_1^0}\J(\tau)\xrightarrow{\hspace{0.1cm}\partial_1\hspace{0.1cm}}W_0\longrightarrow 0,\vspace{0.1cm}\]
and the fact that $|\Delta|$ is homotopic to a  ball (and hence the Euler characteristic of $\Delta$ is equal to 1), we can rewrite (\ref{firstdimformula}) as follows,
\begin{align}
\dim C^r_k(\Delta)&=\sum_{i=0}^3 (-1)^i\hspace{-0.2cm}
\sum_{\beta\in\Delta_{3-i}^0}\hspace{-0.3cm}\dim \mathcal{R}/\mathcal{J}(\beta)_k+ \dim H_{1}(\J)_k - \dim H_{0}(\J)_k\label{withJ}\\
&=\dim R_k+\dim \bigoplus_{\sigma\in\Delta_{2}}\J(\sigma)_k -\dim (W_{1})_k.\label{secondeq}
\end{align}
In the following theorem we collect two results from \cite{Sck}, restricting ourselves to the trivariate case.
\begin{thm}[Schenck, \cite{Sck}]\label{schencktheorem}
Assume that $\Delta$ is a topological $3$-ball.
\begin{enumerate}[(1)]
\item \TT The homology module
$H_i(\calr/\J)$ has dimension $\leq i-1$ for all $i\leq 3$.\label{H_0}
\item The module $C^r(\hat\Delta)$ is free if and only if $H_1(\J)=H_0(\J)=0$. In that case, the Hilbert series of $C^r(\hat\Delta)$ is determined by local data, i.e., by the Hilbert series of the various $\calr/\J(\beta)$, $\beta\in\Delta_i^0$.\label{free-cond}
\end{enumerate}
\end{thm}
\noindent Then, it follows from Theorem \ref{schencktheorem}--(1) that $\dim H_1(\calr/\J)=0$, and hence $\dim H_0(\J)=0$, since these two modules are isomorphic, see (\ref{h1isoh0}) above. Therefore, $H_0(\J)$ vanishes in sufficiently high degree, and so for $k\gg 0$:
\begin{equation}\label{gg}
\dim C^r_k(\Delta)=\sum_{i=0}^3 (-1)^i\hspace{-0.1cm}
\sum_{\beta\in\Delta_{3-i}^0}\dim \mathcal{R}/\mathcal{J}(\beta)_k+ \dim H_{1}(\J)_k.
\end{equation}
In the case that $C^r(\hat\Delta)$ is free, Theorem \ref{schencktheorem}--(2), and formula (\ref{firstdimformula}) imply 
{
\begin{align}
\dim\,C^r_k(\Delta)=\sum_{i=0}^3 (-1)^i
\sum_{\beta\in\Delta_{3-i}^0}\dim \calr/\mathcal{J}(\beta)_k=\dim R_k+\sum_{i=1}^3(-1)^i\dim \hspace{-0.1cm}\bigoplus_{\beta\in\Delta_{3-i}^0}\J(\beta)_k.\label{free}
\end{align}}Let us notice that all the terms in (\ref{free}) only involve
modules generated by powers of linear forms. Before considering bounds
on $\dim C_k^r(\Delta)$ for general $\Delta$ (sections \ref{bounds1}
and  \ref{bounds2}), we will consider the dimension of modules generated by powers of linear forms,
leading so to an approximation for  $\dim C_k^r(\Delta)$ when
$C_k^r(\Delta)$ is free (see Theorem \ref{upperboundfreecase}).

Since by definition $R=\R[x,y,z,w]$, it is easy to see that the space of homogeneous polynomials in $R$ of degree $k$ has dimension
\begin{equation}\label{ringinfourvar}
\dim R_k=\binom{k+3}{3}.
\end{equation}
For  $\sigma\in\Delta_2^0$, the ideal $\J(\sigma)$ is generated by the power $r+1$ of the linear form vanishing on $\hat{\sigma}$, thus
\begin{equation*}
\dim \J(\sigma)_k=\binom{k+3-(r+1)}{3}
\end{equation*}
and hence
\begin{equation}\label{planes}
\dim \bigoplus_{\sigma\in\Delta_2^0} \J(\sigma)_k=f_2^0\,\binom{k+2-r}{3}.
\end{equation}
For $\tau\in\Delta_1^0$, the ideal $\J(\tau)$ by definition is the ideal generated by the powers $r+1$ of the linear forms that define hyperplanes incident to $\hat{\tau}$ in $\hat\Delta$. By construction, each interior edge $\tau$ is at least in the intersection of two (different) hyperplanes corresponding to 2-dimensional faces of $\Delta$. Hence $\J(\tau)$ has at least two generators for any $\tau$.

Let us give a numbering $\tau_1,\dots,\tau_{f_1^0}$ to the elements in $\Delta_1^0$. Note that
for any edge $\tau_i$, we may translate $\tau_i$ to be along one  coordinate axis, and hence may assume that the linear forms in $\J(\tau_i)$ involve only two variables, say $x$ and $y$. Thus,
\[
\calr/\J(\tau_i)\cong \R[z,w]\otimes_\R \R[x,y]/\J(\tau_i)
\]
If $\ell_1$,\dots, $\ell_{s_i}$ are pairwise linearly independent linear forms in $\R[x,y]$ defining $s_i$ hyperplanes incident to $\hat\tau_i$,  and the ideal $\mathcal{J}(\tau_i)$ is generated by $\ell_1^{r+1},\dots ,\ell_{s_i}^{r+1}$,  then the ideal $\J(\tau_i)$  has the following resolution \cite{family}:
{
\begin{equation}\label{resolution}
0\rightarrow {R}(-\Omega_i-1)^{a_i}\oplus {R}(-\Omega)^{b_i}\rightarrow\oplus_{j=1}^{s_i}{R}(-r-1)\rightarrow \J(\tau_i)\rightarrow 0
\end{equation}}where  $\Omega_i$, and the multiplicities $a_i$ and $b_i$ are given by
\begin{equation}\label{formulas}
\Omega_i = \left\lfloor \frac{s_i\,r}{s_i-1}\right\rfloor + 1,\quad a_i=s_i\,(r+1)+(1-s_i)\,\Omega_i,\quad b_i=s_i-1-a_i.
\end{equation}
For each $i$, the number $s_i$ corresponds to the number of different slopes of the hyperplanes incident to $\tau_i$. 

It follows that
{\begin{align}
\dim\bigoplus_{\tau_i\in\Delta_1^0}\J(\tau_i)_k=
\sum_{i=1}^{f_1^0}s_i\binom{k+3-(r+1)}{3}-b_i\binom{k+3-\Omega_i}{3}-a_i\binom{k+3-(\Omega_i+1)}{3}\label{edges}
\end{align}}Here, and throughout the paper we adopt the convention that the binomial coefficient $\binom{u}{m}$ is zero if $u<m$.

Finally, for a vertex $\gamma\in\Delta_0^0$, by definition  $\J(\gamma)$ is the ideal generated by the powers $r+1$ of the linear forms that define hyperplanes incident to $\hat\gamma$. Similarly as before, we give a numbering $\gamma_1,\dots,\gamma_{f_0^0}$ to the vertices in $\Delta_0^0$. Any vertex $\gamma_i$ may be translated to the origin, and hence we may assume that the linear forms generating $\J(\gamma_i)$ involve only three variables, say $x$, $y$ and $z$. In the next section we will discuss some approaches to find the dimension for such ideals. This, together with (\ref{ringinfourvar}-\ref{edges}) will give us an approximation (and in some cases an exact) formula for the dimension of the spline space when $C^r(\hat\Delta)$ is free. We will also use the formulas presented in this section to prove upper and lower bounds in the general settings in Sections \ref{bounds1} and \ref{bounds2}.

\section{On the dimension of the modules $\calr/\J(\gamma)$ }\label{dimJ0}

By translating the vertex $\gamma\in\Delta_0^0$ to the origin, we may assume that the linear forms defining the planes in $\Delta^0$ incident to $\hat\gamma$ involve only the variables $x,y$ and $z$. Thus we have, 
\[
\calr/\J(\gamma)\cong \R[w]\otimes_\R \R[x,y,z]/\J(\gamma).
\]
Let $\mathsf R:=\R[x,y,z]$. From the previous isomorphism, we can study $\dim _\R \calr/\J(\gamma)_k$ by considering the Hilbert function of the $\mathsf R$-module $\mathsf R/\J(\gamma)$. Let us recall that the Hilbert function $H(M)$ of a graded $\mathsf R$-module $M$ is the sequence defined by 
\[H(M,k):=\dim_\R M_k. \]
The problem of computing the Hilbert function associated to ideals of prescribed powers of linear forms, not only in 3, but in $n$ variables, has attracted a great deal of attention in the last years. Its study is linked to classical problems \cite{queen}, and in spite of many partial results (see e.g. \cite{ciliberto} and references therein) it is still open. 

The connection between the Hilbert function of powers of linear forms
and the Hilbert function of a related set of \textit{fat points} in
projective space \cite{iarrobino} has been strongly used to prove
several results in this topic. This connection translates the problem on powers of linear forms into the
study of linear systems in projective $n$-spaces with prescribed
multiplicity at given points. 
The ideal of fat points $I(m; l_{1}, \ldots, l_{r})$ is the ideal of homogeneous
polynomials which vanish at the points $l_{1}, \ldots, l_{t}\in
\pr^{n}$ with multiplicity $m$ (all derivatives of order $m-1$ vanish at the points).
By apolarity, we have 
$$ 
H (\mathsf R/\langle l_1^{r+1},\dots, l_t^{r+1}\rangle, k) = 
H ( I (k-r; l_{1}, \ldots, l_{t}), k).
$$
For instance, in the case of points in $\pr^{1}$, the Hilbert function is given by  the formula of the dimension
obtained from resolution (\ref{resolution}) for ideals generated by power of linear forms in two variables \cite{gs}. 

In this section, we will study the dimension of ideals generated by
powers of linear forms in three variables.
By apolarity, this corresponds to study the Hilbert function of ideals of fat points in $\pr^{2}$. It is clear that if we have $t$ different points in $\pr^2$ then the \textit{expected}
Hilbert function $H ( I (k-r; l_{1}, \ldots, l_{t}), k)$ is given by 
\begin{equation}\label{expdim}
E (t,r+1,3)_{k} := \max \biggl(0,\frac{1}{2} \bigl((k+1)\,(k+2) - t (k-r) (k-r+1\bigr))\biggr).
\end{equation}
Thus, for any linear forms $l_{1}, \ldots, l_{t}$ in three variables, the Hilbert function of $\mathsf R/\langle l_1^{r+1},\dots,
l_t^{r+1}\rangle$  satisfies 
\begin{equation}\label{eq:dimExpected}
H (\mathsf R/\langle l_1^{r+1},\dots, l_t^{r+1}\rangle, k) \geq E (t,r+1,3)_{k}.
\end{equation} 
In \cite{froberg}, Fr\"oberg made a conjecture about the  Hilbert function associated to an ideal generated by a \textit{generic set of forms} in $n$ variables and proved the conjecture for $n=2$. Since then, many authors have studied the conjecture and particularly the special case when the forms generating the ideal are powers of linear equations. The conjecture has been proved in several cases and under certain conditions, see for instance \cites{iarrobino,ciliberto} and the references therein. In particular, for the purpose in this section, it has been proved for ideals generated by generic forms in $n=3$ variables \cite{anick}.

The formula conjectured by Fr\"oberg for the Hilbert function associated to an ideal generated by $t$ forms of degree $r+1$ in a polynomial ring ${R}$ in $n$ variables over $\R$ (or any other field of characteristic zero) will be denoted by $F(t,r+1,n)_i$. This sequence is frequently called Fr\"oberg's sequence and it is defined with the following formula:
\begin{equation}\label{powerseries}F(t,r+1,n)_i=\begin{cases}
F'(t,r+1,n)_i,&\mbox{if \,$F'(t,r+1,n)_u>0$\, for all $u\leq i$,}\\
0&\mbox{otherwise;}
\end{cases}
\end{equation}
where  $F'(t,r+1,n)_i$ is given by
\begin{equation*}
F'(t,r+1,n)_i=\dim_{\R} R_i+\sum_{1\leq j\leq n}(-1)^j\dim_{\R} R_{i-(r+1)j}\binom{t}{j}
\end{equation*}
with the convention that the binomial coefficient $\binom{t}{j}$ is zero if $t<j$. 

\medskip

\noindent We have the following lemma.

\begin{lem}[\cite{iarrobino}]\label{Halgebra}
For any set of different $t$ linear forms $l_1,\dots,l_t$ in $\mathsf R$, and an integer $r\geq 0$, the Hilbert function of the power ideal satisfies:
\begin{equation}\label{inqforlinearforms}
{
\dim_{\R}(\mathsf R/\langle l_1^{r+1},\dots, l_t^{r+1}\rangle)_i\geq F(t,r+1,3)_i\geq E (t,r+1,3)_{i}.}
\end{equation}
Equality holds on the left of (\ref{inqforlinearforms}) when $t\leq 3$, and also when $t=4$ and $l_1,\dots,l_4$ are generic.
\end{lem}
\begin{proof}
Since Fr\"oberg's conjecture is valid for $n=3$ \cite{anick}, then $F(t,r+1,3)_i=H(\mathsf R/\langle f_1,\dots,f_t\rangle,i)$ for generic forms $f_1,\dots, f_t$ of degree $r+1$ in $\mathsf R$. 

Although in general, $\R$-algebras defined by $t$ generic forms of some degree are non isomorphic, they have the same Hilbert function \cite{froberg}, and this Hilbert function is minimal among the Hilbert function of algebras defined by  $t$ forms of the given degree. Thus,  $F(t,r+1,3)_k$ bounds below the Hilbert function $H(\mathsf R/\langle f_1,\dots,f_t\rangle,i)$ where $f_1,\dots,f_t$ are any (non necessarily generic) forms of degree $r+1$ in $\mathsf R$. In particular when $f_1,\dots,f_t$ are powers of linear forms. This implies the inequality on the left of (\ref{inqforlinearforms}). 

The right inequality is clear from definitions \ref{powerseries} and
\ref{expdim} for $F(t,r+1,3)_k$ and $E(t,r+1,3)_k$, respectively 
(since $\dim \mathsf R_{i}= \frac{1}{2}(i+1)\,(i+2)$). 

For  $t\leq 3$, it is the Hilbert function of a complete intersection. The case $t=4$ is a particular case of the result by Stanley \cite{stanley}. 
\end{proof}

\begin{rem} In the settings of Lemma \ref{Halgebra}, when the number of (different) linear forms is $4\leq t\leq 8$, the dimension of the ideal is given by the Fr\"oberg sequence if the points in $\pr^2$ corresponding to the linear forms  are in ``good position" \cites{harbourne1,nagata}. For being in good position, there are some conditions on the divisors on the surface determined by the blow up of the points. For a given set of points in $\pr^2$ those conditions can be verified but there is not a general formula for the dimension that can be given a priori without that verification. 
In his article \cite{harbourne1}, Harbourne also conjectured that the Hilbert function
for ideals generated by powers of any $t\geq 9$ linear forms is given by the Fr\"oberg sequence. 
This  conjecture turned out to be equivalent to other three conjectures, which together gave rise to the  well-celebrated 
Segre-Harbourne-Gimigliano-Hirschowitz Conjecture \cite{ciliberto}.
This conjecture states that $H (\mathsf R/\langle L_1^{r+1},\dots, L_t^{r+1}\rangle, k) = F (t,r+1,3)_{k}$ when $t>8$ and $L_1,\dots,L_t$ are generic linear forms.
It is a special case of the conjecture made by Iarrobino
\cite{iarrobino}, which states that the Fr\"oberg sequence in $\pr^n$
gives the Hilbert function for ideals generated by uniform powers of
generic linear forms except in few cases.  
\end{rem}

\begin{prop}\label{sum-vertices}
Let $\J(\gamma_i)$ be the ideal generated by the powers $r+1$ of the $t_i$ linear forms defining the hyperplanes containing the vertex $\hat\gamma_i$ in $\hat\Delta$, then 
\[\dim \bigoplus_{i=1}^{f_0^0}\mathcal R/\J(\gamma_i)_k\geq \sum_{i=1}^{f_0^0}\biggl(\sum_{j=0}^{k}{F}(t_i,r+1,3)_j\biggr);\]
equality holds if for each vertex $\gamma_i\in\Delta_0^0$, the number $t_{i}$ of generators of $\J(\gamma_i)$ is $t_i= 3$, 
or $t_{i}=4$ and the linear forms are generic.
\end{prop}
\begin{proof}
By translating the vertex $\gamma_i$ for $1\leq i\leq f_0^0$ to the origin, we may assume that $\mathcal{J}(\gamma_i)$ is generated by powers of linear forms in three variables. Thus, by Lemma \ref{Halgebra} we have 
\begin{align*}
\dim \mathcal{R}/\mathcal{J}(\gamma_i)_k&=\dim (\R[w]\otimes_\R \R[x,y,z]/\mathcal{J}(\gamma_i))_k\\
&=\sum_{j=0}^k \dim (\R[x,y,z]/\mathcal{J}(\gamma_i))_j\geq \sum_{j=0}^k F(t_i,r+1,3)_j,
\end{align*}
The proposition follows by applying the previous procedure to each vertex $\gamma_i\in\Delta_0^0$, for $i=1,\dots,f_0^0$.
\end{proof}

\noindent We will use this proposition in the next sections to prove lower and upper bounds on $\dim C_k^r(\Delta)$. 

\begin{rem}
{In \cite{sainverse} and  \cite{MRN}, using also the duality between powers of linear forms and ideals of fat points, the authors relate Fr\"oberg's conjecture to the presence or failure of the \textit{weak Lefschetz property} (if multiplication by a general linear form has or not the maximal rank in every degree). A consequence of this connection, is that the results  about the failure of the  weak Lefschetz property for ideals in $n+1$ variables can be interpreted as results about when an ideal generated by powers of general linear forms  in $n$ variables fails to have the Hilbert function predicted by Fr\"oberg.  
The first theorem concerning the dimension of spline spaces using this approach was originally due to  Stanley \cite{stanley}. He showed that when $t=n+1$,  the Hilbert function of an ideal generated by prescribed powers of $t$ general linear forms in $n$ variables is the same as the Hilbert function conjectured by Fr\"oberg.}

\end{rem}

\begin{rem}{ 
A geometric interpretation of Fr\"oberg--Iarrobino conjecture is given in \cite{chandler2005}. 
A linear system is said to be \textit{special} if it does not have the expected dimension. 
In the planar case $\pr^2$ (number of variables $n=3$), Segre--Harbourne--Gimigliano--Hirschowitz's conjecture describes all special linear systems: a linear system is special if and only if it contains a multiple (-1)-curve in its base locus. In spite of many partial results (see e.g. \cite{ciliberto} and references therein), the conjecture is still open. }
\end{rem}
\begin{rem}
For  $\pr^3$, there is an analogous conjecture formulated by Laface and Ugaglia, see \cite{antonio2012}. The authors employ cubic Cremona trasformations to decrease the degree and the multiplicity of the points.
In the recent article \cite{elisa}, the linear components of the base locus of linear systems in $\pr^n$ are studied and the notion of linear-speciality is introduced: a linear system is linearly non-special if its speciality is only caused by its linear base locus. Sufficient conditions for a linear system to be linearly non-special for arbitrary number of points, and necessary conditions for small numbers of points are given.
\end{rem}

\section{An upper bound on $\dim C_k^r(\Delta)$}\label{bounds1}

In this section we will use formulas (\ref{withJ}--\ref{free}) from Section  \ref{construction} and Proposition \ref{sum-vertices} from Section \ref{dimJ0}, to prove an upper bound on $\dim C_k^r(\Delta)$ for a 3-dimensional simplicial complex $\Delta$. 

Let us establish a numbering $\tau_1,\dots,\tau_{f_1^0}$ on the
interior edges $\tau$ in $\Delta_1^0$. For each $i=1,\dots, f_1^0$,
let ${s}_i$ be (as before) the number of different planes supporting
the faces incident to $\tau_i$, and define
$\tilde{s}_i$ as the number of those planes which correspond to triangles whose other two edges are either on $\partial\Delta$, or have index smaller than $i$.  See Fig.~\ref{countingpic} as an example.

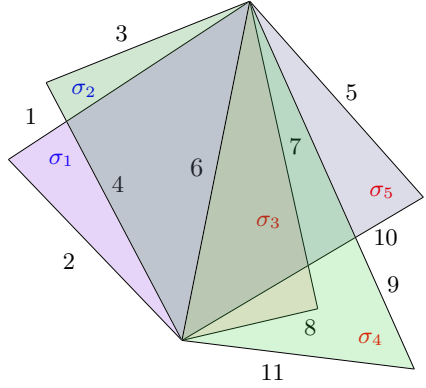
\begin{figure}[!ht]
\scalebox{1}
{\begin{pspicture}(0,-2.465)(6.12875,2.465)
\usefont{T1}{ppl}{m}{n}
\rput(0.3,0.9){\small{$1$}}
\usefont{T1}{ppl}{m}{n}
\rput(0.8,-1.02){\small{$2$}}
\usefont{T1}{ppl}{m}{n}
\rput(1.5,2){\small{$3$}}
\usefont{T1}{ppl}{m}{n}
\rput(1.45,0){\small{$4$}}
\usefont{T1}{ppl}{m}{n}
\rput(4.55,1.225){\small{$5$}}
\usefont{T1}{ppl}{m}{n}
\rput(2.5,0.225){$6$}
\usefont{T1}{ppl}{m}{n}
\rput(3.8,0.5){\small{$7$}}
\usefont{T1}{ppl}{m}{n}
\rput(4,-1.9){\small{$8$}}
\usefont{T1}{ppl}{m}{n}
\rput(5.1,-1.32){\small{$9$}}
\usefont{T1}{ppl}{m}{n}
\rput(5.0,-0.7){\small{$10$}}
\usefont{T1}{ppl}{m}{n}
\rput(3.5,-2.5){\small{$11$}}
\usefont{T1}{ppl}{m}{n}
\rput(0.7,0.325){\large \color{blue}{\small $\sigma_1$}}
\usefont{T1}{ppl}{m}{n}
\rput(1,1.25){\large \color{blue}{\small $\sigma_2$}}
\rput(3.45,-0.5){\large \color{red}{\small $\sigma_3$}}
\usefont{T1}{ppl}{m}{n}
\rput(4.8,-2.05){\large \color{red}{\small $\sigma_4$}}
\usefont{T1}{ppl}{m}{n}
\rput(4.95,-0.1){\large \color{red}{\small$ \sigma_5$}}
\pspolygon[linewidth=0.01cm,linecolor=white,fillstyle=solid,fillcolor=BlueViolet,opacity=0.2](2.3,-2.07)(3.2,2.43)(0.0,0.33)
\psline[linewidth=0.005cm](2.3,-2.07)(0.0,0.33)
\psline[linewidth=0.005cm](0.0,0.33)(3.2,2.43)
\pspolygon
[linewidth=0.01cm,linecolor=white,fillstyle=solid,fillcolor=red,opacity=0.1]
(2.3,-2.07)(3.2,2.43)(4.1,-1.65)
\psline[linewidth=0.005cm](3.2,2.43)(4.1,-1.65)
\psline[linewidth=0.005cm](4.1,-1.65)(2.3,-2.07)
\pspolygon[linewidth=0.01cm,linecolor=white,fillstyle=solid,fillcolor=MidnightBlue,opacity=0.15]
(2.3,-2.07)(3.2,2.43)(5.5,-0.17)
\psline[linewidth=0.005cm](3.2,2.43)(5.5,-0.17)
\psline[linewidth=0.005cm](5.5,-0.17)(2.3,-2.07)
\pspolygon[linewidth=0.01cm,linecolor=white,fillstyle=solid,fillcolor=ForestGreen,opacity=0.2]
(2.3,-2.07)(3.2,2.43)(0.5,1.35)
\psline[linewidth=0.005cm](0.5,1.35)(3.2,2.43)
\psline[linewidth=0.005cm](2.3,-2.07)(0.5,1.35)
\pspolygon[linewidth=0.01cm,linecolor=white,fillstyle=solid,fillcolor=LimeGreen,opacity=0.2]
(2.3,-2.07)(3.2,2.43)(5.38,-2.45)
\psline[linewidth=0.005cm](3.2,2.43)(5.38,-2.45)
\psline[linewidth=0.005cm](5.38,-2.45)(2.3,-2.07)
\psline[linewidth=0.015cm,linecolor=black](3.2,2.43)(2.3,-2.07)
\end{pspicture}}
\caption{For $\tau_6$, ${s_6}=5$ and $\tilde{s_6}=2$.}\label{countingpic}
\end{figure}

We consider the embedding $\hat\Delta$ of $\Delta$ in $\R^4$, and for each edge $\tau_i\in\Delta_1^0$ we define the ideals  $\J(\tau_i)$ and $\widetilde{\J}(\tau_i)$ in $R=\R[x,y,z,w]$, to be the ideal generated by the power $r+1$ of the  $s_i$, and $\tilde s_i$  linear forms of hyperplanes incident to $\hat\tau_i$, respectively.

\begin{thm}\label{upperbound}
The dimension of $C_k^r(\Delta)$ is bounded above by
{
\begin{align*}
\dim C_k^r(\Delta)&\leq \binom{k+3}{3}+f_2^0\binom{k+2-r}{3}\\
&-\sum_{i=1}^{f_1^0}\biggl[\tilde s_i\binom{k+2-r}{3}-\tilde b_i\binom{k+3-\tilde \Omega_i}{3}-\tilde a_i\binom{k+2-\tilde \Omega_i}{3}\biggr]
\end{align*}}with $\tilde{\Omega}_i= \left\lfloor \frac{\tilde s_i\,r}{\tilde s_i-1}\right\rfloor + 1,\; \tilde a_i=\tilde s_i\,(r+1)+(1-\tilde s_i)\,\tilde\Omega_i,\; \tilde b_i=\tilde s_i-1-\tilde a_i$ \;
if $\tilde s_i>1$, and $\tilde a_i=\tilde b_i=\tilde\Omega_i=0$ when $\tilde s_i=1$.
\end{thm}

\begin{proof}
Let us consider the map
\[\delta_1:\bigoplus_{\sigma=(\tau,\tau',\tau'')\in \Delta^{0}_{2}}\hspace{-0.2cm}\J(\sigma)[\sigma]\to\bigoplus_{\tau_i\in\Delta_1^0}\bigoplus_{\;\sigma\in N(\tau_i)}\calr[\sigma|\tau_i]\]
where, for each $i$, $N(\tau_i)$ denotes the set of triangles that
contain the edge $\tau_i$ and $([\sigma])_{\sigma\in \Delta_{2}^{0}}$, and 
$([\sigma|\tau_i])_{\sigma \in N(\tau_i)}$ are the canonical bases of
the corresponding free modules.
The map $\delta_1$ is induced by the boundary map $\partial_2$. Thus,  $\delta_1([\sigma])=[\sigma|\tau]-[\sigma|\tau']+[\sigma|\tau'']$ for $\sigma=(\tau,\tau',\tau'')\in\Delta^0_2$, see Fig.~\ref{triangoriented}.
\begin{figure}[!ht]
\hspace{0.5cm}\scalebox{1.5} 
{
\hspace{0.6cm}\begin{pspicture}(0,-0.99421877)(10.42,0.99421877)
\psline[linewidth=0.0001cm,arrowsize=0.1cm 2.0,arrowlength=1.4,arrowinset=0.4]{->}(1.771875,0.8957813)(1.18,0.01)
\psline[linewidth=0.02cm](0.771875,-0.6042187)(1.771875,0.8957813)
\psline[linewidth=0.0005cm](0.771875,-0.6042187)(1.18,0.01)
\rput{-128.69786}(2.9449465,1.3898537){\psarc[linewidth=0.02,arrowsize=0.05291667cm 2.0,arrowlength=1.4,arrowinset=0.4]{->}(1.8061883,-0.012178499){0.24895173}{353.69788}{315.0}}
\psline[linewidth=0.0001cm,arrowsize=0.1cm 2.0,arrowlength=1.4,arrowinset=0.4]{->}(0.771875,-0.6042187)(1.771875,-0.6042187)
\psline[linewidth=0.02cm](0.771875,-0.6042187)(2.771875,-0.6042187)
\psline[linewidth=0.0005cm](1.771875,-0.6042187)(2.771875,-0.6042187)
\psline[linewidth=0.0001cm,arrowsize=0.1cm 2.0,arrowlength=1.4,arrowinset=0.4]{->}(2.771875,-0.6042187)(2.271875,0.1457813)
\psline[linewidth=0.02cm](2.771875,-0.6042187)(1.771875,0.8957813)
\psline[linewidth=0.0005cm](2.271875,0.1457813)(1.771875,0.8957813)
\usefont{T1}{ppl}{m}{n}
\rput(1.1,0.79578125){\small $\sigma$}
\usefont{T1}{ptm}{m}{n}
\rput(4.8,0.09578125){\tiny $\partial_2[\sigma]=[\tau]-[\tau']+[\tau'']$}
\usefont{T1}{ppl}{m}{n}
\rput(1.0,0.09578125){\tiny $\tau'$}
\usefont{T1}{ppl}{m}{n}
\rput(2.6,0.09578125){\tiny $\tau$}
\usefont{T1}{ppl}{m}{n}
\rput(1.7,-0.78){\tiny $\tau''$}
\end{pspicture}
}
\caption{Orientation of a triangle $\sigma\in\Delta_2^0$.}\label{triangoriented}
\end{figure}
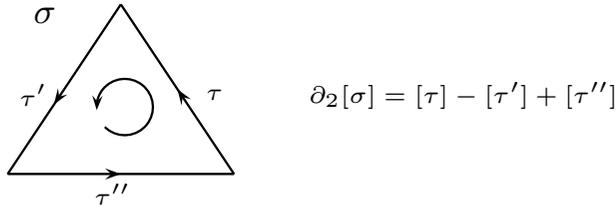

\noindent Let
\[\varphi_1:\bigoplus_{\tau_i\in\Delta_1^0}\bigoplus_{\sigma\in N(\tau_i)}\calr[\sigma|\tau_i]\to\bigoplus_{\tau\in\Delta_1^0}\calr[\tau]\]
with
\begin{equation*}
\varphi_1([\sigma|\tau_i])=
\begin{cases}
\TT [\tau_i]\;\text{ if}\;\tau_i\in\Delta_1^0\\
0\;\text{ if}\; \tau_i\notin \Delta_1^0.
\end{cases}
\end{equation*}
Then, for the restriction map $\partial_2$ to the ideals $\J(\sigma)$  in the complex (\ref{Jcomplex}), we have $\partial_2=\varphi_1\circ\delta_1$. We consider the map
\[\pi_1:\bigoplus_{\tau_i\in\Delta_1^0}\bigoplus_{\sigma\in N(\tau_i)}\calr[\sigma|\tau_i]\to \bigoplus_{\tau_i\in\Delta_1^0}\bigoplus_{\sigma\in N(\tau_i)}\calr[\sigma|\tau_i]\]
defined as follows, according to the numbering established on the
edges. For each triangle $\sigma=(\tau,\tau',\tau'')\in\Delta_2^0$,
either one or two of the edges of $\sigma$ are in $\partial\Delta$, or
$\tau,\tau',\tau''\in\Delta^0_1$. By construction, at least one of the
edges of $\sigma$ is in the interior of $\Delta$, and hence there is
an index assigned to it. Without loss of generality, we may assume
that $\tau\in\Delta_1^0$ is the edge with the smallest index among the
edges of $\sigma$ that are in the interior $\Delta_1^0$. Then
$\pi_{1}$ is defined for the edges corresponding to $\sigma$ by:
\begin{itemize}
 \item $\pi_1([\sigma|\tau])=[\sigma|\tau]$, 
 \ST\item $\pi_1([\sigma|\tau'])=\pi_1([\sigma|\tau''])=0$.
\end{itemize}
Let us denote $\tilde{\partial}_2:=\varphi_1\circ\pi_1\circ\delta_1$.

For $\tau_i\in\Delta_1^0$, define $\tilde N(\tau_i)$ as the set of  triangles $\sigma\in\Delta_2^0$ that contain $\tau_i$ as an edge and whose other two edges do not have index bigger than $i$. 

Thus, $\widetilde\J(\tau_i) = \sum_{\sigma\in\tilde{N}(\tau_i)} \calr \ell^{r+1}_\sigma\subseteq \J(\tau_i)$. By construction, and using the notation we introduced in Section \ref{construction}, we have 
\[\widetilde{W}_{1}:=\Im \tilde{\partial}_2=\bigoplus_{\tau_i\in\Delta_1^0}\widetilde{\J}(\tau_i)[\tau_i].\]
Therefore, $\dim (W_{1})_{k}:=\dim \Im(\partial_2)_k\geq \dim \Im(\tilde{\partial}_2)_k$. 
From formula (\ref{secondeq}) for $\dim C_k^r(\Delta)$ in Section \ref{construction}, it follows
\[\dim C_k^r(\Delta)\leq \dim \calr_k+ \dim \bigoplus_{\sigma\in\Delta_{2}}\J(\sigma)_k-\dim (\widetilde{W}_{1})_{k}.\]
By a change of coordinates such that the edge $\tau_i$ is along one of the coordinate axis, we may assume that the linear forms in $\widetilde\J(\tau_i)$ only involve two variables and then use the resolution (\ref{resolution}) for ideals generated by power of linear forms in two variables to get a formula for $\dim \widetilde W_{1_k}$. Thus, we get
{\small
\begin{align*}
\dim (\widetilde{W}_{1})&_{k}=\dim \bigoplus_{\tau_i\in\Delta_1^0}\widetilde{\J}(\tau_i)=\sum_{i=1}^{f_1^0}\tilde s_i\binom{k+2-r}{3}-\tilde b_i\binom{k+3-\tilde \Omega_i}{3}-\tilde a_i\binom{k+2-\tilde \Omega_i}{3},
\end{align*}}with $\tilde{s}_i=|\tilde{N}(\tau_i)|$, $\tilde\Omega_i$, $\tilde a_i$ and $\tilde b_i$ given by formulas (\ref{formulas}), with $\tilde s_i$ instead of $s_i$. This together with formulas (\ref{ringinfourvar})  and (\ref{planes}) prove the theorem.
\end{proof}

\noindent A different upper bound can be proved for $\dim C_k^r(\Delta)$ when $C_k^r(\hat\Delta)$ is free, i.e.~when $H_1(\J)=H_0(\J)=0$ and  the formula for $\dim C_k^r(\Delta)$ reduces to (\ref{free}).

We keep the numbering on the edges $\tau_i\in\Delta_1^0$, and establish also a numbering $\gamma_1,\dots,\gamma_{f_0^0}$ on the interior vertices of $\Delta$. For each $i=1,\dots,f_0^0$, let $t_i$ be the number of linear forms  defining the hyperplanes containing the vertex $\hat\gamma_i$ in $\hat\Delta$, and $\J(\gamma_i)$ the ideal generated by the power $r+1$ of these linear forms.

Using the results from Section \ref{dimJ0}, Fr\"oberg's sequence gives a formula to bound from above $\dim \J(\gamma)_k$ for each $\gamma\in\Delta_0^0$, and we have the following.

\begin{thm}\label{upperboundfreecase}
If $C^r(\Delta)$ is free then,
dimension of $C_k^r(\Delta)$ is bounded above by
{
\begin{align}
\dim \,C_k^r(\Delta)&\leq \binom{k+3}{3}+f_2^0\binom{k+2-r)}{3}\nonumber\\
&-\sum_{i=1}^{f_1^0}\biggl[s_i\binom{k+2-r)}{3}- b_i\binom{k+3- \Omega_i}{3}- a_i\binom{k+2-\Omega_i}{3}\biggr]\nonumber\\
&+f_0^0\binom{k+3}{3}-\sum_{i=1}^{f_0^0} \biggl(\sum_{j=0}^k F(t_i,r+1,3)_j\biggr),\nonumber
\end{align}}with \,$s_i$ and $t_i$ as defined above, $\Omega_i = \left\lfloor \frac{s_i\,r}{s_i-1}\right\rfloor + 1,\; a_i=s_i\,(r+1)+(1-s_i)\,\Omega_i, \;  b_i=s_i-1-a_i$,
 and $F(t_i,r+1,3)_j$  the $j$-th term of Fr\"oberg's sequence associated to an ideal generated by the power $r+1$ of $t_i$ forms in three variables.
\end{thm}
\begin{proof}

Formulas (\ref{ringinfourvar}), (\ref{planes}) and (\ref{edges}) give
the dimension for the modules in (\ref{free}) corresponding to the
tetrahedra, triangles and edges of $\Delta$. For the last term, which
corresponds to the vertices, we apply Proposition \ref{sum-vertices},
and obtain the formula appearing in the last line of the bound in the
statement. It corresponds to applying Fr\"oberg's sequence  (\ref{powerseries}) to ideals  generated by powers of linear forms in three variables in the ring $R=\R[x,y,z,w]$.
\end{proof}

\begin{rem}\label{freecol}
From Proposition \ref{sum-vertices}, when $C^r(\Delta)$ if free, the upper bound on dimension $\dim C_k^r(\Delta)$ in the previous theorem can be improved depending on the number of different planes containing the vertices in $\Delta^0$. For instance, if the number of different planes incident to $\gamma_i$ is $t_i=3$ for every $\gamma_i\in\Delta_0^0$, then Fr\"oberg's sequence gives the exact dimension for the ideal corresponding to the vertices and thus, $\dim C_k^r(\Delta)$ is exactly given by the formula in Theorem 
\ref{upperboundfreecase}, see Examples \ref{oct1} and \ref{ct} below. Also in the case $t_i=3$, this upper bound coincides with  the formula for the lower bound that we  prove in Theorem \ref{lowbound} in the next section.
\end{rem}

\section{A lower bound on $\dim C_k^r(\Delta)$}\label{bounds2}
Let us recall formula (\ref{withJ}) for $\dim C_k^r(\Delta)$ from Section (\ref{construction}),
\[
\dim C^r_k(\Delta)=\dim \calr_k+\sum_{i=1}^3 (-1)^i\hspace{-0.04cm}\dim\hspace{-0.1cm}\bigoplus_{\beta\in\Delta_{3-i}^0}\hspace{-0.1cm}\J(\beta)_k +\dim H_{1}(\J)_k- \dim H_{0}(\J)_k.
\]
If we take zero as a lower bound for $\dim H_1(\J)_k$, then for any $k\geq 0$:
{\begin{equation}\label{forlowbound}
\dim C^r_k(\Delta)\geq\dim \calr_k+\sum_{i=1}^2 (-1)^i\hspace{-0.04cm}\dim\hspace{-0.1cm}\bigoplus_{\beta\in\Delta_{3-i}^0}\hspace{-0.1cm}\J(\beta)_k+\dim (W_0)_k
\end{equation}}where
$W_0:=\Im(\partial_1)$, as defined before. From
(\ref{ringinfourvar}), (\ref{planes}) and (\ref{edges}), we have
explicit expressions for all the terms in (\ref{forlowbound}) except
for $\dim\, (W_0)_k$. By numbering the vertices in $\Delta_0^0$ and by
applying the analogous to the procedure used in last section, we are
going to get an explicit formula that approximates $\dim\, (W_0)_k$ from below. This, by (\ref{forlowbound}), immediately leads to a lower bound on $\dim C_k^r(\Delta)$.

Let us fix the ordering $\gamma_1,\dots,\gamma_{f_0^0}$ on the vertices in $\Delta_0^0$. For each vertex $\gamma_i$, denote by $M(\gamma_i)$  the set of edges $\tau$ in $\Delta_1^0$ that contain the vertex $\gamma_i$.
Let $\tilde M(\gamma_i)$ be the set of interior edges connecting $\gamma_i$ to one of the first $i-1$ vertices in the list, or to a vertex in the boundary.

For each $\gamma_i\in\Delta_0^0$, let $t_i$ be defined as before,
the number of generators of $\J(\gamma_i)$. Define the ideal $\widetilde\J(\gamma_i)$ as
\[
\widetilde\J(\gamma_i)=\langle \ell_\sigma^{r+1}\rangle_{\sigma\ni\tau}  \text{\quad for $\tau\in\tilde M({\gamma_i})$},
\]
and let $\tilde t_i$ be the number of generators of $\widetilde \J(\gamma_i)$. Finally define $\zeta_i=\min(3,\tilde t_i)$.
\begin{thm}\label{lowbound}
The dimension $\dim C_k^r(\Delta)$ is bounded below by
\begin{align}
\dim C^r_k(\Delta)&\geq\binom{k+3}{3}+\biggl[f_2^0\binom{k+2-r}{3}\label{lowerboundformula}\\
&-\sum_{i=1}^{f_1^0}\biggl[s_i\binom{k+2-r}{3}- b_i\binom{k+3- \Omega_i}{3}- a_i\binom{k+2-\Omega_i}{3}\biggr]\nonumber\\
&+f_0^0\binom{k+3}{3}-\sum_{i=1}^{f_0^0}\biggl(\sum_{j=0}^k F(\zeta_i,r+1,3)_j\biggl)\biggr]_+\nonumber
\end{align}
with \,$s_i$ the number of different planes incident to $\tau_i$, $\Omega_i = \left\lfloor \frac{s_i\,r}{s_i-1}\right\rfloor + 1$,  $a_i=s_i\,(r+1)+(1-s_i)\,\Omega_i$,  $b_i=s_i-1-a_i$ (\ref{formulas}), and $\zeta_i=\min(3,\tilde t_i)$.
\end{thm}
\begin{proof}

Consider the following map
\[\delta_0:\bigoplus_{\tau=(\gamma,\gamma')}\J(\tau)[\tau]\to\bigoplus_{\gamma_i\in\Delta_0^0}\bigoplus_{\tau\in M(\gamma_i)}\calr[\tau|\gamma_i]\]
such that $\delta_0$ is induced by the boundary map $\partial_1$, so that $\delta_0([\tau])=[\tau|\gamma]-[\tau|\gamma']$ for $\tau=(\gamma,\gamma')\in\Delta_1^0$. Let $\varphi_0$ be the map defined as
\[\varphi_0:\bigoplus_{\gamma_i\in\Delta_0^0}\bigoplus_{\tau\in M(\gamma_i)}\calr[\tau|\gamma_i]\to\bigoplus_{\gamma_i\in\Delta_0^0}\calr[\gamma_i]\]
with
\begin{equation*}
\varphi_0([\tau|\gamma_i])=
\begin{cases}
\TT [\gamma_i]\;\text{if}\;\gamma_i\in\Delta_0^0\\
0\;\text{if}\;\gamma_i\notin\Delta_0^0\,.
\end{cases}
\end{equation*}
Then, for the restriction of the map $\partial_1$ to the ideals $\J(\gamma_i)$ in the complex (\ref{Jcomplex}), $\partial_1=\varphi_0\circ\delta_0$. Consider the map
\[\pi_0:\bigoplus_{\gamma_i\in\Delta_0^0}\bigoplus_{\tau\in M(\gamma_i)}\calr[\tau|\gamma_i]\to\bigoplus_{\gamma_i\in\Delta_0^0}\bigoplus_{\tau\in M(\gamma_i)}\calr[\tau|\gamma_i]\]
defined as follows, according to the numbering established on  $\Delta_0^0$. For an edge $\tau=(\gamma,\gamma')\in\Delta_1^0$, at least one of the vertices $\gamma$ or $\gamma'$ is in $\Delta_0^0$. Let us assume $\gamma\in\Delta_0^0$, and either $\gamma'$ is in $\partial\Delta$, or $\gamma'\in\Delta_1^0$ and the index of $\gamma$ is smaller than the index of $\gamma'$. Then $\pi_0$ is defined on the vertices of $\tau$ by: 
\begin{itemize}
\item  $\pi_0([\tau|\gamma])=[\tau|\gamma]$, 
\ST\item $\pi_0([\tau|\gamma'])=0$.
\end{itemize}
We define the map $\tilde\partial_1$ by \[\tilde\partial_1:=\varphi_0\circ\pi_0\circ\delta_0.\] For each $\gamma_i\in\Delta_0^0$,  $\widetilde\J(\gamma_i)=\sum_{\tau\in\tilde M(\gamma_i)}\sum_{\sigma\ni\tau}\calr\ell_\sigma^{r+1}\subseteq\J(\gamma_i)$. Then, by construction
\[\Im(\tilde\partial_1)=\bigoplus_{i=1}^{f_0^0}\widetilde\J(\gamma_i)[\gamma_i],\]
and therefore $\dim(W_0)_k:=\dim(\Im \partial_1)_k\geq\dim(\Im\tilde\partial_1)_k$. Thus, from (\ref{forlowbound}) it follows that
\begin{equation*}
\dim C^r_k(\Delta)\geq\dim \calr_k+\sum_{i=1}^2 (-1)^i\hspace{-0.04cm}\dim\hspace{-0.1cm}\bigoplus_{\beta\in\Delta_{3-i}^0}\hspace{-0.1cm}\J(\beta)_k+\dim \bigoplus_{i=1}^{f_0^0}\widetilde\J(\gamma_i)_k.
\end{equation*}
By construction $\tilde t_i\leq t_i$. Choose $\zeta_i$ linear forms from the generators of $\widetilde\J(\gamma_i)$, where  $\zeta_i=\min(3,\tilde t_i)$. Let $\J_{\zeta_i}(\gamma_i)$ be the ideal generated by the powers $r+1$ of these $\zeta_i$ linear forms. Then,
for each $\gamma_i\in\Delta_0^0$, $\J_{\zeta_i}(\gamma_i)\subseteq \widetilde\J(\gamma_i)$, and hence
\[\dim\bigoplus_{i=1}^{f_0^0}\widetilde\J(\gamma_i)\geq\dim \bigoplus_{i=1}^{f_0^0}\J_{\zeta_i}(\gamma_i)_k.\]
From Proposition \ref{sum-vertices}, we have
\[\sum_{j=0}^k F(\zeta_i,r+1,3)_j=\dim \calr/\J_{\zeta_i}(\gamma_i)_k,\] 
where $ F(\zeta_i,r+1,3)_j$ is defined by the formula \eqref{powerseries}. 
Thus, we obtain the lower bound on the dimension of $C_k^r(\Delta)$ given in the statement of the theorem. Since the dimension of the spline space is at least the number of polynomials in tree variables of degree less than or equal to $k$, then we take the positive part of the additional terms.  
\end{proof}
\noindent The next corollary follows directly from the proof of the previous theorem.
\begin{cor}\label{cor} For a fixed numbering on the interior vertices and $\widetilde\J(\gamma_i)$ defined as above, 
\[\dim H_0(\J)\leq \dim \bigoplus_{\gamma_i\in\Delta_0^0}\J(\gamma_i)-\dim\sum_{\gamma_i\in\Delta_0^0}\widetilde\J(\gamma_i).\]
\end{cor}

\begin{rem}{
Following the proof of Theorem \ref{lowbound}, better lower bounds can be proved if the linear forms defining the ideals $\J(\gamma_i)$ are generic, see Proposition \ref{sum-vertices} and the remarks at the end of Section \ref{dimJ0}. By knowing the Hilbert function of ideals generated by powers of $\tilde{t}_i\geq 4$ linear forms in three variables one might avoid the step of taking $\zeta_i=\min (3,\tilde{t}_i)$ and improve the lower bound.}
\end{rem}

\begin{rem}{
In the case of $C^r(\hat\Delta)$ being free, we can use the upper
bounds either from Theorem \ref{upperboundfreecase} or Theorem \ref{upperbound}, together with the lower bound in Theorem \ref{lowbound}. Depending on the value of $k$ and $r$, they provide a closer approximation to the exact dimension, see the examples in Section \ref{examples}.}
\end{rem}

\section{Examples}\label{examples}
For the central configurations that we will consider in this section, it is easy to see that $H_0(\J)$ is always zero; it can be deduced directly, or it can be proved using the construction in the last section, see Corollary \ref{cor}. The  values for $H_1(\mathcal J)$ were computed using the Macaulay2 software \cite{macaulay}.

\begin{ex}\label{oct1}{Let $\Delta$ be a octahedron subdivided into eight tetrahedra by placing a symmetric central vertex, see Fig.~\ref{octahedron1}.}
\end{ex}

\begin{figure}[!ht]
\scalebox{1.5} 
{\begin{pspicture}(0,-2.01)(4.01,2.01)
\pspolygon[linewidth=0.01cm,linecolor=white,fillstyle=solid,fillcolor=red,opacity=0.17]
(1.8,-0.1)(1.9,-2.0)(1.1,0.2)
\pspolygon[linewidth=0.01cm,linecolor=white,fillstyle=solid,fillcolor=red,opacity=0.17]
(2.45,-0.4)(1.9,-2.0)(1.8,-0.1)
\pspolygon[linewidth=0.01cm,linecolor=white,fillstyle=solid,fillcolor=yellow,opacity=0.2]
(2.45,-0.4)(1.9,2.0)(1.8,-0.1)
\pspolygon[linewidth=0.01cm,linecolor=white,fillstyle=solid,fillcolor=yellow,opacity=0.2]
(1.8,-0.1)(1.9,2.0)(1.1,0.2)
\pspolygon[linewidth=0.01cm,linecolor=white,fillstyle=solid,fillcolor=yellow,opacity=0.10]
(1.8,-0.1)(2.45,-0.4)(4.0,0.0)
\pspolygon[linewidth=0.01cm,linecolor=white,fillstyle=solid,fillcolor=yellow,opacity=0.1]
(1.8,-0.1)(1.1,0.2)(4.0,0.0)
\pspolygon[linewidth=0.01cm,linecolor=white,fillstyle=solid,fillcolor=yellow,opacity=0.18]
(0.0,0.0)(1.1,0.2)(1.8,-0.1)
\pspolygon[linewidth=0.01cm,linecolor=white,fillstyle=solid,fillcolor=yellow,opacity=0.18]
(0.0,0.0)(2.45,-0.4)(1.8,-0.1)
\pspolygon[linewidth=0.01cm,linecolor=white,fillstyle=solid,fillcolor=MidnightBlue,opacity=0.18]
(1.8,-0.1)(1.9,2.0)(0.0,0.0)
\pspolygon[linewidth=0.01cm,linecolor=white,fillstyle=solid,fillcolor=MidnightBlue,opacity=0.18]
(1.9,-2.0)(1.8,-0.1)(4.0,0.0)
\pspolygon[linewidth=0.01cm,linecolor=white,fillstyle=solid,fillcolor=Green,opacity=0.11]
(1.8,-0.1)(1.9,2.0)(4.0,0.0)
\pspolygon[linewidth=0.01cm,linecolor=white,fillstyle=solid,fillcolor=Green,opacity=0.11]
(0.0,0.0)(1.9,-2.0)(1.8,-0.1)
\pspolygon[linewidth=0.01cm,linecolor=white,fillstyle=solid,fillcolor=Green,opacity=0.11]
(0.0,0.0)(1.9,-2.0)(2.45,-0.4)
\pspolygon[linewidth=0.01cm,linecolor=white,fillstyle=solid,fillcolor=Green,opacity=0.11]
(2.45,-0.4)(1.9,-2.0)(4.0,0.0)
\pspolygon[linewidth=0.01cm,linecolor=white,fillstyle=solid,fillcolor=red,opacity=0.12]
(1.1,0.2)(1.9,2.0)(0.0,0.0)
\pspolygon[linewidth=0.01cm,linecolor=white,fillstyle=solid,fillcolor=red,opacity=0.10]
(1.1,0.2)(1.9,2.0)(4.0,0.0)
\pspolygon[linewidth=0.01cm,linecolor=white,fillstyle=solid,fillcolor=red,opacity=0.08]
(2.45,-0.4)(1.9,2.0)(4.0,0.0)
\psline[linewidth=0.01cm,linecolor=black](0.0,0.0)(1.9,-2.0)
\psline[linewidth=0.015cm,linecolor=black](1.9,-2.0)(2.45,-0.4)
\psline[linewidth=0.015cm,linecolor=black](2.45,-0.4)(0.0,0.0)
\psline[linewidth=0.01cm,linecolor=black](1.9,-2.0)(4.0,0.0)
\psline[linewidth=0.015cm,linecolor=black](2.45,-0.4)(4.0,0.0)
\psline[linewidth=0.01cm,linecolor=black](1.9,2.0)(4.0,0.0)
\psline[linewidth=0.015cm,linecolor=black](2.45,-0.4)(1.9,2.0)
\psline[linewidth=0.01cm,linecolor=black](1.9,2.0)(0.0,0.0)
\usefont{T1}{pcr}{m}{n}
\rput(1.85,-0.95){\tiny $1$}
\usefont{T1}{pcr}{m}{n}
\rput(2.15,-0.2){\tiny $2$}
\rput(2.8,-0.1){\tiny $3$}
\rput(1.35,0.1){\tiny $4$}
\rput(0.9,-0.05){\tiny $5$}
\rput(1.85,0.75){\tiny $6$}
\end{pspicture}
}
\caption{Regular octahedron.}\label{octahedron1}
\end{figure}
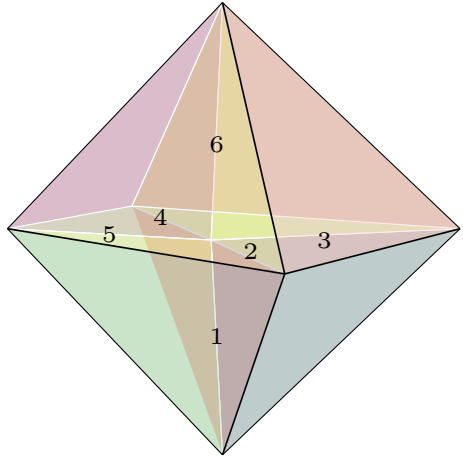
Computations show that $H_1(\J)$ is zero for all non-generic constructions  \cite{Sck}. Since in this partition, there are exactly three different planes through the central vertex, then the Fr\"oberg sequence gives us an explicit formula for the dimension of the ideal associated to the (unique) interior vertex, and the dimension  $\dim C_k^r(\Delta)$ is given by the upper bound formula in Theorem (\ref{upperboundfreecase}), see Remark \ref{freecol}. The formula is the following
\begin{align*}
\dim C_k^r(\Delta)=& \binom{k+3}{3}+12\binom{k+3-(r+1)}{3}\\
&-\sum_{i=1}^{6}\biggl[2\binom{k+3-(r+1)}{3}-\binom{k+3-(2r+2)}{3}\biggr]\\
&+\binom{k+3}{3}-\sum_{j=0}^k F(3,r+1,3)_j.
\end{align*}
From the definition of Fr\"oberg's sequence (\ref{powerseries}), 
\begin{align*}
F(3,r+1,3)_j=
\binom{j+2}{2}-3\binom{j+1-r}{2}+3\binom{j-2r}{2}-\binom{j-3r-1}{2}.
\end{align*}
It is easy to check that  $F(3,r+1,3)_j> 0$ for every $0\leq j< 3r+3$, and equal to zero otherwise. Hence, we can write
\begin{equation}\label{3v}
\sum_{j=0}^k F(3,r+1,3)_j=
\binom{k+3}{3}-3\binom{k+2-r}{3}+3\binom{k-2r+1}{3}-\binom{k-3r}{3}
\end{equation}
and thus, the formula for the dimension of the spline space on the regular octahedron in Fig.~\ref{octahedron1} reduces to the expression
{\begin{align*}\label{dimoct1}
\dim C_k^r(\Delta)=& \binom{k+3}{3}+3\binom{k+2-r}{3}+
3\binom{k+1-2r}{3}+\binom{k-3r}{3}.
\end{align*}}

\begin{ex}\label{ej2}{
Let us consider the generic case of an octahedron subdivided into tetrahedra, where no set of four vertices of the octahedron is coplanar, Fig.~\ref{octahedron2}. As we mentioned above, we have $H_0(\J)$ equal to zero. But in contrast to the regular case, $H_1(\J)$ is equal to zero when $r=1$ but not for any other value of $r$ \cite{Sck}.}
\end{ex}
\begin{figure}[!ht]
\scalebox{1.5} 
{\begin{pspicture}(0,-1.9391599)(3.6900234,1.9391599)
\pspolygon[linewidth=0.01,linecolor=white,fillstyle=solid,fillcolor=red,opacity=0.17](1.6,0.5)(2.6,1.9316598)(3.8,0.21829198)
\pspolygon[linewidth=0.01,linecolor=white,fillstyle=solid,fillcolor=red,opacity=0.17](1.6,0.5)(2.6,1.9316598)(0.1,0.3)
\pspolygon[linewidth=0.01,linecolor=white,fillstyle=solid,fillcolor=yellow,opacity=0.2](2.1,0.1)(2.7,-1.8)(1.6,0.5)
\pspolygon[linewidth=0.01,linecolor=white,fillstyle=solid,fillcolor=yellow,opacity=0.2](2.5,-0.6)(2.7,-1.8)(2.1,0.1)
\pspolygon[linewidth=0.01,linecolor=white,fillstyle=solid,fillcolor=yellow,opacity=0.10](2.5,-0.6)(2.6,1.9316598)(2.1,0.1)
\pspolygon[linewidth=0.01,linecolor=white,fillstyle=solid,fillcolor=yellow,opacity=0.1](2.1,0.1)(2.6,1.9316598)(1.6,0.5)
\pspolygon[linewidth=0.01,linecolor=white,fillstyle=solid,fillcolor=yellow,opacity=0.18](2.1,0.1)(2.5,-0.6)(3.8,0.21829198)
\pspolygon[linewidth=0.01,linecolor=white,fillstyle=solid,fillcolor=yellow,opacity=0.18](2.1,0.1)(1.6,0.5)(3.8,0.21829198)
\pspolygon[linewidth=0.01,linecolor=white,fillstyle=solid,fillcolor=MidnightBlue,opacity=0.18](0.1,0.3)(1.6,0.5)(2.1,0.1)
\pspolygon[linewidth=0.01,linecolor=white,fillstyle=solid,fillcolor=MidnightBlue,opacity=0.18](0.1,0.3)(2.5,-0.6)(2.1,0.1)
\pspolygon[linewidth=0.01,linecolor=white,fillstyle=solid,fillcolor=Green,opacity=0.11](2.1,0.1)(2.6,1.9316598)(0.1,0.3)
\pspolygon[linewidth=0.01,linecolor=white,fillstyle=solid,fillcolor=Green,opacity=0.11](2.7,-1.8)(2.1,0.1)(3.8,0.21829198)
\pspolygon[linewidth=0.01,linecolor=white,fillstyle=solid,fillcolor=Green,opacity=0.11](2.1,0.1)(2.6,1.9316598)(3.8,0.21829198)
\pspolygon[linewidth=0.01,linecolor=white,fillstyle=solid,fillcolor=Green,opacity=0.11](0.1,0.3)(2.7,-1.8)(2.1,0.1)
\pspolygon[linewidth=0.01,linecolor=white,fillstyle=solid,fillcolor=red,opacity=0.12](1.6,0.5)(2.7,-1.8)(3.8,0.21829198)
\pspolygon[linewidth=0.01,linecolor=white,fillstyle=solid,fillcolor=red,opacity=0.10](1.6,0.5)(2.7,-1.8)(0.1,0.3)
\pspolygon[linewidth=0.01,linecolor=white,fillstyle=solid,fillcolor=red,opacity=0.08](2.5,-0.6)(2.6,1.9316598)(3.8,0.21829198)
\psline[linewidth=0.01cm](0.1,0.3)(2.7,-1.8)
\psline[linewidth=0.015cm](2.7,-1.8)(2.5,-0.6)
\psline[linewidth=0.015cm](2.5,-0.6)(0.1,0.3)
\psline[linewidth=0.01cm](2.7,-1.8)(3.8,0.21829198)
\psline[linewidth=0.015cm](2.5,-0.6)(3.8,0.21829198)
\psline[linewidth=0.01cm](2.6,1.9316598)(3.8,0.21829198)
\psline[linewidth=0.015cm](2.5,-0.6)(2.6,1.9316598)
\psline[linewidth=0.01cm](2.6,1.9316598)(0.1,0.3)
\usefont{T1}{pcr}{m}{n}
\rput(2.45,-1.0){\tiny $1$}
\usefont{T1}{pcr}{m}{n}
\rput(2.35,-0.3){\tiny $2$}
\rput(2.8,0.15){\tiny $3$}
\rput(1.85,0.3){\tiny $4$}
\rput(1.0,0.2){\tiny $5$}
\rput(2.3,0.8){\tiny $6$}
\end{pspicture}
}
\caption{Generic octahedron.}\label{octahedron2}
\end{figure}
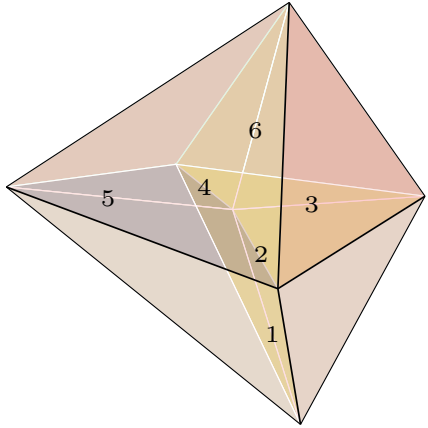

For this partition $\Delta$, we have $t=12$ different planes corresponding to the triangles meeting at the central vertex. Then $\zeta=\min (3,t)=3$, and using the formula (\ref{3v}) from the previous example for the sum of the $F(3,r+1,3)_j$ for $r=1$, Theorem (\ref{lowbound}) gives us the following lower bound
\begin{align*}
\ST  \dim C_k^1(\Delta)&\geq\binom{k+3}{3}+\biggl[12\binom{k+1}{3}\hspace{-0.1cm}-\hspace{-0.1cm}6\biggl[3\binom{k+1}{3}-2\binom{k}{3}\biggr]+\binom{k+3}{3}\hspace{-0.1cm}-\sum_{j=0}^k F(3,2,3)_j\biggr]_+\ST\\
\ST  &=\binom{k+3}{3}+\biggl[-3\binom{k+1}{3}+12\binom{k}{3}-3\binom{k-1}{3}+\binom{k-3}{3}\biggr]_+.
\end{align*}
In order to find an upper bound, we apply Theorem \ref{upperbound} for some ordering on the interior edges of the partition. For instance,
with the numbering on the edges as in Fig.~\ref{octahedron2}, we have $\tilde s_1=0$, $\tilde s_2=1$, $\tilde s_3=\tilde s_4=2$, $\tilde s_5=3$, and $\tilde s_6=4$, and so for any degree $k$:

\[
\dim C_k^1(\Delta)\leq \binom{k+3}{3}+\binom{k+1}{3}+4\binom{k}{3}+2\binom{k-1}{3}.
\]
Also, when $r=1$, since $H_1(\J)=0$, we can find an upper bound by applying Theorem \ref{upperboundfreecase}. This upper bound is given by the formula
\[
\dim C_k^1(\Delta)\leq  2\binom{k+3}{3}+6\binom{k-1}{3}
-\sum_{j=0}^k F(12,2,3)_j.
\]
From (\ref{powerseries}), we have that $F(12,2,3)_j>0$ only for $j=0,1$, and it is zero otherwise. Then
\[
\dim C_k^1(\Delta)\leq\begin{cases}
\ST \binom{k+3}{3}&\mbox{when $k=0,1$}\\
\TT 2\binom{k+3}{3}+6\binom{k-1}{3}-4&\mbox{for $k\geq 2$}.
\end{cases}
\]
\begin{rem}
In \cite{jimmy}, by using inverse systems of fat points, the author studies the dimension of $C^2$ splines on tetrahedral complexes in $\R^3$ sharing a single interior vertex. By a classification of fat point ideals, the question in this case leads to analyze  ideals associated  to (only) $\leq 10$ hyperplanes passing through a common vertex.  
\end{rem}

\begin{ex}\label{ct}{Let $\Delta$ be the Clough--Tocher split consisting of a tetrahedron which has been split about an interior point into four subtetrahedra, Fig.~\ref{CT-split}.}
\end{ex}
\begin{figure}[!ht]
\scalebox{1.5} 
{\begin{pspicture}(0,-1.3)(3.6900234,1.9391599)
\pspolygon[linewidth=0.02,linecolor=white,fillstyle=solid,fillcolor=yellow,opacity=0.15]
(2.2,0.2)(0.3,0.8)(2.9,-1.2)
\pspolygon[linewidth=0.02,linecolor=white,fillstyle=solid,fillcolor=Orange,opacity=0.15]
(2.2,0.2)(3.7,0.5)(2.9,-1.2)
\pspolygon[linewidth=0.02,linecolor=white,fillstyle=solid,fillcolor=yellow,opacity=0.2]
(2.2,0.2)(3.7,0.5)(0.3,0.8)
\pspolygon[linewidth=0.01,linecolor=white,fillstyle=solid,fillcolor=Green,opacity=0.12](0.3,0.8)(2.2,0.2)(2.2,1.8)
\pspolygon[linewidth=0.01,linecolor=white,fillstyle=solid,fillcolor=Green,opacity=0.15](3.7,0.5)(2.2,0.2)(2.2,1.8)
\pspolygon[linewidth=0.01,linecolor=white,fillstyle=solid,fillcolor=Green,opacity=0.12](2.2,0.2)(2.9,-1.2)(2.2,1.8)
\pspolygon[linewidth=0.01,linecolor=white,fillstyle=solid,fillcolor=Blue,opacity=0.06](3.7,0.5)(2.9,-1.2)(0.3,0.8)
\pspolygon[linewidth=0.01,linecolor=white,fillstyle=solid,fillcolor=Blue,opacity=0.06](3.7,0.5)(2.2,1.8)(0.3,0.8)
\pspolygon[linewidth=0.01,linecolor=white,fillstyle=solid,fillcolor=Blue,opacity=0.06](3.7,0.5)(2.2,1.8)(2.9,-1.2)
\pspolygon[linewidth=0.01,linecolor=white,fillstyle=solid,fillcolor=Blue,opacity=0.06](2.9,-1.2)(2.2,1.8)(0.3,0.8)
\psdots[dotsize=0.05](2.2,0.2)
\psline[linewidth=0.01cm,linestyle=dashed,dash=0.05cm 0.04cm](0.3,0.8)(3.7,0.5)
\psline[linewidth=0.015cm](2.2,1.8)(3.7,0.5)
\psline[linewidth=0.015cm](2.9,-1.2)(2.2,1.8)
\psline[linewidth=0.015cm](2.2,1.8)(0.3,0.8)
\psline[linewidth=0.015cm](0.3,0.8)(2.9,-1.2)
\psline[linewidth=0.015cm](3.7,0.5)(2.9,-1.2)
\end{pspicture}
\usefont{T1}{pcr}{m}{n}
\rput(-2.3,1.7){\tiny $4$}
\rput(-1.65,2.4){\tiny $2$}
\rput(-0.9,1.7){\tiny $1$}
\rput(-1.3,0.9){\tiny $3$}
}
\caption{Clough--Tocher split.}\label{CT-split}
\end{figure}
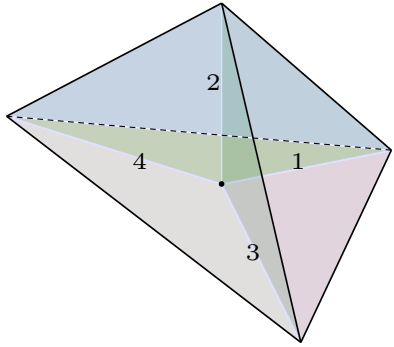
\noindent We consider  $r=1$ and $r=2$. In these two cases the homology module $H_1(\mathcal{J})$ is zero. 

\begin{enumerate}[(i)]
\item For $r=1$, as in the previous example, we have
{\small
\begin{align*}
\sum_{j=0}^k F(3,2,3)_j&=\binom{k+3}{3}-3\binom{k+1}{3}+3\binom{k-1}{3}-\binom{k-3}{3}
\end{align*}}Then, the lower bound on the spline space proved in Theorem \ref{lowbound} is given by
{\small
\begin{align*}
&\dim C^1_k(\Delta)\geq\binom{k+3}{3}+\biggl[-3\binom{k+1}{3}
+8\binom{k}{3}-3\binom{k-1}{3}+\binom{k-3}{3}\biggr]_+.
\end{align*}
}The upper bound we obtained in this example, by applying Theorem \ref{upperbound} with the numbering of the edges as in Fig.~\ref{CT-split} is the following:
\begin{align*}
&\dim C^1_k(\varDelta)\leq\binom{k+3}{3}+\binom{k-1}{3}+2\binom{k}{3}.
\end{align*}
Since for $r=1$ the homology module $H_1(\mathcal{J})=0$, applying Theorem \ref{upperboundfreecase} we find the upper bound
\begin{equation}\label{up}
\dim C_k^1(\Delta)\leq\begin{cases}
\ST 1&\mbox{for $k=0$}\\
\TT 2\binom{k+3}{3}-6\binom{k+1}{3}+8\binom{k}{3}-4&\mbox{for $k\geq 1$}.
\end{cases}
\end{equation}

\medskip
\noindent The formula (\ref{up}) coincides with the generic dimension formula computed in \cite{ASW} for this partition $\Delta$. 
Although the formula in \cite{ASW} holds only for $k\geq 8$ (and $r=1$), it in turn coincides with the lower bound formula proved in \cite{boundsTriva} in every degree $k\geq 0$. In fact, in general, the dimension of the spline space of any nongeneric decomposition is always greater than or equal to the generic dimension, it is the smallest dimension encountered as one moves the vertices of the complex. 
Thus, since the lower bound formula proved in \cite{boundsTriva}
coincides with  the upper bound we proved above (\ref{up}), we deduce the
following result:

\noindent
{\em  the exact dimension of the $C^1$ spline space over the Clough--Tocher split is 

$
\dim C_k^1(\Delta)=\begin{cases}
\ST 1&\mbox{for $k=0$}\\
\TT 2\binom{k+3}{3}-6\binom{k+1}{3}+8\binom{k}{3}-4&\mbox{for $k\geq 1$}.
\end{cases}
$
}
\medskip

\begin{rem}
In 
\cite{tatjana}, the authors consider the general case of this example. They study $C^1$ splines on the $n$-dimensional Clough-Tocher split, i.e., on a simplex in $\R^n$ partitioned around an interior point into $n+1$ subsimplices. A formula for the dimension is proved by combining results about the module structure of the spline space and Bernstein-B\'ezier methods.
\end{rem}

\item Let us consider the case $r=2$.

\noindent A lower bound is given by the formula 
\[\dim C_k^2(\Delta)\geq \binom{k+3}{3}+\biggl[ -3\binom{k}{3}+4\binom{k-1}{3}+4\binom{k-2}{3}-3\binom{k-3}{3}+\binom{k-6}{3}\biggr]_+\]
Using that $H_1(\mathcal J)=0$, and Theorem \ref{upperboundfreecase}, the following is an upper bound for $k\geq 3$:
\[\dim C_k^2(\Delta)\leq 2\binom{k+3}{3}-6\binom{k}{3}+4\binom{k-1}{3}+4\binom{k-2}{3}-14\]
\end{enumerate}
The values of the previous bounds on $\dim C_k^2(\Delta)$ for $k\leq 9$ are given in the following table. The first row shows the values obtained using the lower bound formula from \cite{boundsTriva}.

\medskip
\medskip

{
\hspace{0.7cm}\begin{tabular}{|l|r|r|r|r|r|r|r|r|r|}
  \hline
  $k$&\quad 1&\quad 2&\quad 3&\quad 4&\quad 5&\quad 6&\quad 7&\quad 8&\quad 9 \\
  \hline\hline
  Lower bound \cite{boundsTriva} & {\small 4} & {\small 10} &{\small  20} & {\small 35} &{\small  56} &{\small  84} &{\small  120} & {\small 179} & {\small 261}\\
 Lower bound & {\small 4} & {\small 10} & {\small 20} & {\small 35} & {\small 56} & {\small 84} & {\small 123} & {\small 187} & {\small 282}\\  
 Upper bound & {\small 4} & {\small 10} & {\small 20} & {\small 36} & {\small 58} & {\small 90} & {\small 136} & {\small 200} & {\small 286}\\  [-0.15ex]
 
  \hline
  
\end{tabular}
}

\medskip
\medskip

\begin{rem}The examples above illustrate the improvement that our lower and upper bounds provide with respect to previous results in the literature.  Furthermore, as we showed in the last example, the formulas we presented here might be combined with results obtained by using different techniques leading thus to sharper bounds, and in many cases to the exact dimension of the space. 
\end{rem}

\bibliography{mourrain-villamizar}

\subsection*{Acknowledgment}
The two authors would like to acknowledge the
support of the EU–FP7 Initial Training Network SAGA: ShApes, Geometry and Algebra (2008-2012).

\end{document}